\documentclass{amsart}[10pt]
\usepackage{enumerate, amsfonts,pstricks,epsfig,amssymb, amsmath, amsthm}

\newlength{\tabwidth}
\newlength{\tabheight}
\newlength{\tabrule}
\newlength{\tabwidthx}
\newlength{\tabheightx}

\def\gentabbox#1#2#3#4{\vbox to \tabheight{\setlength{\tabrule}{#3}%
  \setlength{\tabwidthx}{#1\tabwidth}\addtolength{\tabwidthx}{\tabrule}%

\setlength{\tabheightx}{#2\tabheight}\addtolength{\tabheightx}{-\tabheight}%
  \hbox to #1\tabwidth{%
 \hspace{-0.5\tabrule}\rule{\tabrule}{#2\tabheight}\hspace{-\tabrule}%
    \vbox to #2\tabheight{\hsize=\tabwidthx%
      \vspace{-0.5\tabrule}\hrule width\tabwidthx height\tabrule%
      \vspace{-0.5\tabrule}\vfil%
      \hbox to \tabwidthx{\hss#4\hss}%
        \vfil\vspace{-0.5\tabrule}%
      \hrule width\tabwidthx height\tabrule\vspace{-0.5\tabrule}}%
 \hspace{-\tabrule}\rule{\tabrule}{#2\tabheight}\hspace{-0.5\tabrule}}%
  \vspace{-\tabheightx}}}
\def\genblankbox#1#2{\vbox to \tabheight{\vfil\hbox to
#1\tabwidth{\hfil}}}
\def\tabbox#1#2#3{\gentabbox{#1}{#2}{0.4pt}{\strut #3}}

\catcode`\:=13 \catcode`\.=13 \catcode`\;=13
\catcode`\>=13 \catcode`\^=13
\def:#1\\{\hbox{$#1$}}
\def.#1{\tabbox{1}{1}{$#1$}}
\def>#1{\tabbox{2}{1}{$#1$}}
\def^#1{\tabbox{1}{2}{$#1$}}
\def;{\genblankbox{1}{1}\relax}
\catcode`\:=12 \catcode`\.=12 \catcode`\;=12
\catcode`\>=12 \catcode`\^=7

\newenvironment{tableau}{\bgroup\catcode`\:=13 \catcode`\.=13
  \catcode`\;=13 \catcode`\>=13 \catcode`\^=13
  \setlength{\tabheight}{3ex}\setlength{\tabwidth}{3ex}%
  \def\b##1##2##3{\gentabbox{##1}{##2}{1.2pt}{\vbox{##3}}}%
  \def\n##1##2##3{\gentabbox{##1}{##2}{0.4pt}{\vbox{##3}}}%
  \vbox\bgroup\offinterlineskip}{\egroup\egroup}

\newtheorem{theorem}{Theorem}[section]
\newtheorem{corollary}[theorem]{Corollary}
\newtheorem{lemma}[theorem]{Lemma}
\newtheorem{proposition}[theorem]{Proposition}
\newtheorem{conjecture}[theorem]{Conjecture}

\newtheorem*{theorem:main}{Theorem \ref{theorem:generators}}
\newtheorem*{lemma:technical2}{Lemma \ref{lemma:technical2}}
\newtheorem*{prop:amongspecial}{Proposition \ref{prop:amongspecial}}
\newtheorem*{prop:makespecial}{Proposition \ref{prop:makespecial}}

\newtheorem*{conjecture*}{Conjecture}

\theoremstyle{definition}
\newtheorem{definition}[theorem]{Definition}

\theoremstyle{remark}

\newtheorem{example}[theorem]{Example}

\newcommand{\aaa}{{\boldsymbol{a}}}

\begin{document}
\title{Module structure of cells in unequal parameter Hecke algebras}
\author{Thomas Pietraho}
\email{tpietrah@bowdoin.edu} 
\address{Department of Mathematics\\Bowdoin College\\Brunswick,
Maine 04011}

\begin{abstract}  A conjecture of C.~Bonnaf\'e, M.~Geck, L.~Iancu, and T.~Lam parameterizes Kazhdan-Lusztig left cells for unequal parameter Hecke algebras in type $B_n$ by  families of standard domino tableaux of arbitrary rank.  Relying on a family of properties outlined by G.~Lusztig and the recent work of C.~Bonnaf\'e, we verify the conjecture and describe the structure of each cell as a module for the underlying Weyl group.
\end{abstract}

\maketitle

\section{Introduction}

Consider a Coxeter system $(W,S)$, a positive weight function $L$, and the corresponding generic Iwahori-Hecke algebra $\mathcal{H}$.   As detailed by G.~Lusztig in  \cite{lusztig:unequal}, a choice of weight function gives rise to a partition of $W$ into left, right, and two-sided Kazhdan-Lusztig cells, each of which carries the structure of an $\mathcal{H}$- as well as a $W$-module.
The cell decomposition of $W$ is understood for all finite Coxeter groups and all choices of weight functions with the exception of type $B_n$. We focus our attention on this remaining case and write $W=W_n$. A weight function is then specified by a choice of two integer parameters $a$ and $b$ assigned to the simple reflections
in $W_n$:

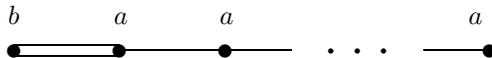
\begin{center}
\begin{picture}(300,30)
\put( 50 ,10){\circle*{5}} \put( 50, 8){\line(1,0){40}} \put( 50,
12){\line(1,0){40}} \put( 48, 20){$b$} \put( 90 ,10){\circle*{5}}
\put(130 ,10){\circle*{5}} \put(230 ,10){\circle*{5}} \put( 90,
10){\line(1,0){40}} \put(130, 10){\line(1,0){25}} \put(170,
10){\circle*{2}} \put(180, 10){\circle*{2}} \put(190,
10){\circle*{2}} \put(205, 10){\line(1,0){25}} \put( 88 ,20){$a$}
\put(128, 20){$a$} \put(222, 20){$a$}
\end{picture}
\end{center}

Given $a,b \neq 0$, we write $s=\tfrac{b}{a}$ for
their quotient. We have the following description of cells due to C.~Bonnaf{\'e}, M.~Geck, L.~Iancu, and T.~Lam.  It is stated in terms of a family of generalized Robinson-Schensted algorithms $G_r$ which define bijections between $W_n$ and same-shape pairs of domino tableaux of rank $r$.

\begin{conjecture*}[\cite{bgil}]\label{conjecture:bgil2}
Consider a Weyl group $W_n$ of type $B_n$
with a weight function $L$ and parameter $s$ defined as above.
    \begin{enumerate}
        \item When $s \not\in \mathbb{N}$, let $r= \lfloor s \rfloor$.  Two elements of $W_n$ lie in the same Kazhdan-Lusztig left cell whenever they share the same right tableau in the image of $G_r$.
        \item  When $s \in \mathbb{N}$, let $r= s-1$.  Two elements of $W_n$ lie in the same
                Kazhdan-Lusztig left cell whenever their right tableaux in the image of $G_r$ are related by moving through a set of non-core open cycles.
    \end{enumerate}
\end{conjecture*}

Significant progress has been made towards the verification of the above, which we detail in Section \ref{section:cellsB}.  Most recently, C.~Bonnaf\'e has shown that if a certain family of statements conjectured by G.~Lusztig is assumed to hold, then the conjecture holds if $s\not \in \mathbb{N}$, and  furthermore, if $s \in \mathbb{N}$, then Kazhdan-Lusztig left cells are unions of the sets described \cite{bonnafe:knuth}.  We sharpen this result, and verify that the conjecture holds in the latter case as well.

We concurrently describe the structure of Kazhdan-Lusztig left cells as $W_n$-modules.
The canonical parameter set for irreducible $W_n$-modules consists of ordered pairs of partitions $(d,f)$ where the the parts of $d$ and $f$ sum to $n$.  As detailed in Section \ref{section:irreducible}, there is a natural identification of this parameter set with the set of partitions $\mathcal{P}_r(n)$ of a fixed rank $r$.  Since $\mathcal{P}_r(n)$ corresponds exactly to the shapes of rank $r$ domino tableaux, the parametrization of Kazhdan-Lusztig left cells via standard tableaux of fixed rank in the above conjecture suggests a module structure for each cell for every choice of weight function. Mainly, the irreducible constituents of the module carried by each cell should correspond to the shapes of the rank $r$ tableaux of its elements, with $r$ determined from the parameter $s$ as in the conjecture.  We verify that this suggested module structure is indeed the one carried by each cell.

Our approach is based on M.~Geck's characterization of left cells as constructible representations; that is, those representations which are obtained by successive truncated parabolic induction and tensoring with the sign representation, see \cite{geck:constructible}.  In Section \ref{section:hecke}, we detail the general construction of Kazhdan-Lusztig cells in an unequal parameter Hecke algebra and extend a result of G.~Lusztig on the intersection of left and right cells to the unequal parameter setting.  In Section \ref{section:typeB}, we detail the situation in type $B_n$ and the relevant combinatorics.  Section \ref{section:constructible} examines constructible representations and provides a combinatorial description of truncated parabolic induction and tensoring with sign, mimicking the work of W.~M.~McGovern in the equal parameter case \cite{mcgovern:leftcells}.  The final section contains the proof of the main results.


\section{Unequal Parameter Hecke Algebras}\label{section:hecke}

We briefly recount the definitions of unequal parameter Hecke algebras and the corresponding Kazhdan-Lusztig cells, following \cite{lusztig:unequal}.

\subsection{Kazhdan-Lusztig Cells}\label{section:klcells}
Consider  a Coxeter system $(W,S)$ and let $\ell$ be the usual  length function.  A {\it weight function} $L:W \rightarrow \mathbb{Z}$ satisfies $L(xy)=L(x)+L(y)$ whenever $\ell(xy)=\ell(x)+\ell(y)$ and is uniquely determined by its values on $S$.  We will consider those weight functions which take positive values on all $s \in S$.

Let $\mathcal{H}$ be the generic
Iwahori-Hecke algebra over $\mathcal{A}= \mathbb{Z}[v, v^{-1}]$ with
parameters $\{v_s \, | \, s \in S\}$, where  $v_x = v^{L(x)}$ for
all $x \in W$. The algebra $\mathcal{H}$ is free over $\mathcal{A}$
and has a basis $\{T_x \, | \, x \in W\}$. Multiplication in $\mathcal{H}$ takes the form
$$T_s T_x = \left\{
        \begin{array}{ll}
            T_{sx} & \text{if $\ell(sx) > \ell(x)$, and}\\
            T_{sx}+(v_s-v_s^{-1}) T_x & \text{if $\ell(sx) < \ell(x)$}
        \end{array}
        \right.
        $$
As in \cite{lusztig:unequal}(5.2), it is possible to construct a Kazhdan-Lusztig basis of $\mathcal{H}$
which we denote by $\{C_x \; | \; x \in W\}$. In terms of it, multiplication has the form
$$ C_x C_y = \sum_{z \in W} h_{xyz} C_z.$$
for some $h_{xyz} \in \mathcal{A}$.  Although we suppress it in the notation, all of these notions depend on the specific choice of weight function $L$.

\begin{definition}
Fix $(W,S)$  a Coxeter system with a weight
function $L$.
We will write $y \leq_{\mathcal{L}} x$ if there exists $s \in S$ such that $C_{y}$ appears with a non-zero coefficient in $C_s C_x$.  By taking the transitive closure, this binary relation defines a preorder on $W$ which we also denote by $\leq_{\mathcal{L}}$.
Let $y \leq_{\mathcal{R}} x$ iff $y^{-1} \leq_{\mathcal{L}} x^{-1}$ and define $\leq_{\mathcal{LR}}$ as the pre-order generated
by $\leq_{\mathcal{L}}$ and $\leq_{\mathcal{R}}$.
\end{definition}
Each of the above preorders defines equivalence relations which we denote by $\sim_\mathcal{L},$ $\sim_\mathcal{R}$, and $\sim_\mathcal{LR}$ respectively.  The resulting equivalence classes are called the {\it left, right, and two-sided Kazhdan-Lusztig cells} of $W$.

As described in \cite{lusztig:unequal}(8.3), Kazhdan-Lusztig cells carry representations of $\mathcal{H}$.  If $\mathfrak{C}$ is a
Kazhdan-Lusztig left cell and $x \in \mathfrak{C}$, then define
$$[\mathfrak{C}]_\mathcal{A} = \bigoplus_{y \leq_\mathcal{L} x} \mathcal{A} C_{y} \Big/ \bigoplus_{y \leq_\mathcal{L} x, y \notin \mathfrak{C} } \mathcal{A} C_{y}.$$

This is a quotient of two left ideals in $\mathcal{H}$ and consequently is itself a
left $\mathcal{H}$-module; it does not depend on the specific choice of $x \in \mathfrak{C}$, is
free over $\mathcal{A}$, and has a basis $\{e_x \; | \; x \in
\mathfrak{C}\}$ indexed by elements of $\mathfrak{C}$ with $e_x$
the image of $C_x$ in the above quotient.
The action of  $\mathcal{H}$ on $[\mathfrak{C}]_\mathcal{A}$ is determined by $$C_x e_y = \sum_{z \in
\mathfrak{C}} h_{xyz} e_z$$ for $x \in W$ and $y \in \mathfrak{C}$. A Kazhdan-Lusztig left cell gives rise to a $W$-module $[\mathfrak{C}]$ by restricting $[\mathfrak{C}]_\mathcal{A}$ to scalars.  The same construction can be used to define module structures on the right and two-sided cells of $W$.

\subsection{A Family of Properties} The main results of this paper rely on a family of conjectures formulated by G.~Lusztig in \cite[\S 14]{lusztig:unequal}.  In the equal parameter case, that is when $L$ is a multiple of the length function $\ell$, a number of results about Kazhdan-Lusztig cells depend on positivity results derived by geometric methods of intersection cohomology.  Unfortunately, this positivity does not hold for unequal parameter Hecke algebras; for examples see \cite[\S 6]{lusztig:leftcells} and \cite[2.7]{geck:leftcells}.
As a substitute,  G.~Lusztig detailed a list properties which both, axiomatize known equal-parameter results, and outline methods of approaching non-positivity in general.  

In order to list Lusztig's conjectures, we must first define two integer-valued functions on $W$.
For any $z \in W$, let $\aaa(z)$ be the smallest non-negative integer so that $h_{xyz} \in v^{\aaa(z)} \mathbb{Z}[v^{-1}]$ for every $x$ and $y$ in $W$ and write $\gamma_{xyz^{-1}}$ for the constant term of $v^{-\aaa(z)} h_{xyz}$.  If $p_{xy}$ is defined by $C_y = \sum_{x\in W} p_{xy} T_x$, then \cite{lusztig:unequal}(5.4) shows that $p_{1z}$ is non-zero.  We write
    $$p_{1z} = n_z v^{-\Delta(z)} + \text{ terms of smaller degree in $v$}$$
thereby defining a constant $n_z$ and integer $\Delta(z)$ for every $z\in W$.
Finally, let
$$\mathcal{D} = \{z \in W \mid \aaa(z)=\Delta(z)\}.$$
Lusztig has conjectured the following statements are true in the general setting of unequal parameter Hecke algebras:

\begin{itemize}
\item[\bf P1.] For any $z\in W$ we have $\aaa(z)\leq \Delta(z)$.
\item[\bf P2.] If $d \in \mathcal{D}$ and $x,y\in W$ satisfy $\gamma_{x,y,d}\neq 0$,
then $x=y^{-1}$.
\item[\bf P3.] If $y\in W,$ there exists a unique $d\in \mathcal{D}$ such that
$\gamma_{y^{-1},y,d}\neq 0$.
\item[\bf P4.] If $z'\leq_{\mathcal{L}\mathcal{R}} z$ then $\aaa(z')\geq \aaa(z)$. Hence, if
$z'\sim_{\mathcal{L}\mathcal{R}} z$, then $\aaa(z)=\aaa(z')$.
\item[\bf P5.] If $d\in \mathcal{D}$, $y\in W$, $\gamma_{y^{-1},y,d}\neq 0$, then
$\gamma_{y^{-1},y,d}=n_d=\pm 1$.
\item[\bf P6.] If $d\in \mathcal{D}$, then $d^2=1$.
\item[\bf P7.] For any $x,y,z\in W$, we have $\gamma_{x,y,z}=\gamma_{y,z,x}$.
\item[\bf P8.] Let $x,y,z\in W$ be such that $\gamma_{x,y,z}\neq 0$. Then
$x\sim_{\mathcal{L}} y^{-1}$, $y \sim_{\mathcal{L}} z^{-1}$, and $z\sim_{\mathcal{L}} x^{-1}$.
\item[\bf P9.] If $z'\leq_{\mathcal{L}} z$ and $\aaa(z')=\aaa(z)$, then $z'\sim_{\mathcal{L}}z$.
\item[\bf P10.] If $z'\leq_{\mathcal{R}} z$ and $\aaa(z')=\aaa(z)$, then $z'\sim_{\mathcal{R}}z$.
\item[\bf P11.] If $z'\leq_{\mathcal{L}\mathcal{R}} z$ and $\aaa(z')=\aaa(z)$, then
$z'\sim_{\mathcal{L}\mathcal{R}}z$.
\item[\bf P12.] Let $I\subseteq S$ and $W_I$ be the parabolic subgroup
defined by $I$. If $y\in W_I$, then $\aaa(y)$ computed in terms of $W_I$
is equal to $\aaa(y)$ computed in terms of $W$.
\item[\bf P13.] Any left cell $\mathfrak{C}$ of $W$ contains a unique element
$d\in \mathcal{D}$. We have $\gamma_{x^{-1},x,d}\neq 0$ for all $x\in \mathfrak{C}$.
\item[\bf P14.] For any $z\in W$, we have $z \sim_{\mathcal{L}\mathcal{R}} z^{-1}$.
\item[\bf P15.] If $v'$ is an indeterminate and $h'_{xyz}$ is obtained from $h_{xyz}$ via the substitution $v \mapsto v'$, then     whenever $\aaa(w)=\aaa(y)$, we have
    $$\sum_{y'} h'_{wx'y'}h_{xy'y}=\sum_{y'} h_{xwy'}h'_{y'x'y}.$$

\end{itemize}

The statements {\bf P1-P15} are known to hold for finite Weyl groups in the equal parameter case by work of Kazhdan-Lusztig \cite{kazhdan:lusztig:schubert} and Springer \cite{springer:intersection}.  If the Coxeter system is of type $I_2(m)$, $H_3$, or $H_4$, they follow from work of Alvis \cite{alvis:left} and DuCloux \cite{ducloux:positivity}.  In the unequal parameter case, {\bf P1-P15} have been verified by Geck in types $I_2(m)$ and $F_4$ \cite{geck:remarks}, and in the so-called asymptotic case of type $B_n$ by Geck-Iancu \cite{geck:iancu} and Geck \cite{geck:relative}, \cite{geck:remarks}.  Although the geometric approach from which the above follow in the equal parameter case is not available in the general unequal parameter case, it seems that it may not be required.  At least in type $A$, Geck has shown that {\bf P1-P15} hold via elementary, purely algebraic methods \cite{geck:murphy}.

\subsection{The Asymptotic Ring $J$} \label{section:asymptotic}

The goal of this section is to verify Lemma 12.15 of \cite{lusztig:characters} in our more general setting.  We begin with a brief discussion of Lusztig's ring $J$ which can be viewed as an asymptotic version of $\mathcal{H}.$  Although originally defined in the equal parameter case, its construction also makes sense in the setting of unequal parameter Hecke algebras under the the assumption that the conjectures {\bf P1-P15} hold.  Using the methods developed in \cite{lusztig:unequal}, $J$ provides us with a way of studying the left-cell representations of $\mathcal{H}$.

Recall the integers $\gamma_{xyz}$ defined for all $x,$ $y,$ and $z$ in $W$ as the constant terms of $v^{\aaa(z)} h_{xyz^{-1}}.$  Then $J$ is the free abelian group with basis $\{t_x \; | \; x \in W\}$.  To endow it with a ring structure,  define a bilinear product on $J$ by
$$t_x \cdot t_y = \sum_{z \in W} \gamma_{xyz} t_{z^{-1}}$$
for $x$ and $y$ in $W$.  Conjectures {\bf P1-P15} allow us to state the following results.
\begin{theorem}[\cite{lusztig:unequal}] Assuming conjectures {\bf P1-P15}, the following hold:
 \begin{enumerate}
    \item $J$ is an associative ring with identity element $1_J = \sum_{d \in \mathcal{D}} n_d t_d.$
    \item The group algebra $\mathbb{C}[W]$ is isomorphic as a $\mathbb{C}$-algebra to $J_\mathbb{C} = \mathbb{C} \otimes_\mathbb{Z} J.$
 \end{enumerate}
\end{theorem}

Following  \cite[\S 20.2]{lusztig:unequal}, we will write $E_\spadesuit$ for the $J_\mathbb{C}$-module corresponding to a $\mathbb{C}[W]$-module $E$.  It shares its underlying space with $E$, while the action of an element of $J_\mathbb{C}$ is defined by the action of its image under the isomorphism with $\mathbb{C}[W]$.
Consider a left cell $\mathfrak{C}$ of $W$ and define $J^\mathfrak{C}_\mathbb{C}$ to be $\oplus_{x \in \mathfrak{C}} \mathbb{C} t_x.$  By {\bf P8}, this is a left ideal in $J_\mathbb{C}.$  Furthermore,

\begin{theorem}[\cite{lusztig:unequal}]\label{theorem:lusztig2} Assuming that the conjectures {\bf P1-P15} hold, the $J_\mathbb{C}$-modules
$J^\mathfrak{C}_\mathbb{C}$ and $[\mathfrak{C}]_\spadesuit$ are isomorphic.
\end{theorem}

We are ready to address Lemma 12.15 of \cite{lusztig:characters}.  Its original proof relies on a characterization of left cells in terms of the dual bases $\{C_x\}$ and $\{D_x\}$ stated in  \cite{lusztig:characters}(5.1.14). This result in turn relies on positivity properties which do not hold in the unequal parameter case and therefore a new approach to the lemma is required.  We owe the idea of using $J$ in the present proof to M.~Geck.

\begin{lemma}\label{lemma:12.15} Assume that conjectures {\bf P1-P15} hold.  If $\mathfrak{C}$ and $\mathfrak{C'}$ are two left cells in $W$ with respect to a weight function $L$, then
        $$\dim \textup{Hom}_W([\mathfrak{C}],[\mathfrak{C'}]) = | \mathfrak{C} \cap \mathfrak{C'}^{-1} |.$$
\end{lemma}

\begin{proof}

Let $x \in \mathfrak{C}^{-1} \cap \mathfrak{C'}$ and define a map $\phi_x$ on $J^\mathfrak{C}_\mathbb{C}$ via $\phi_x(t_y) = t_y t_x.$  With $x$ and $y$ as above, we can write
$$t_y t_x = \sum \gamma_{yxz} t_{z^{-1}}.$$ For $\gamma_{yxz} \neq 0$,  {\bf P8} implies $x \sim_\mathcal{L} z^{-1}.$ Since $x \in \mathfrak{C}'$, this forces $t_y t_x$ to lie in $J^\mathfrak{C'}_\mathbb{C},$ and we have in fact defined a map $\phi_x: J^\mathfrak{C}_\mathbb{C} \rightarrow J^\mathfrak{C'}_\mathbb{C}.$

We will show that as $x$ runs over the set $\mathfrak{C}^{-1} \cap \mathfrak{C'}$, the maps $\phi_x$ are linearly independent.  So assume that for some constants $a_x$ we have
$$\sum_{x \in \mathfrak{C}^{-1} \cap \mathfrak{C'}} a_x \phi_x = 0 \text{ and, consequently } \sum_{x \in \mathfrak{C}^{-1} \cap \mathfrak{C'}} a_x t_y t_x = 0$$
for all $y \in \mathfrak{C}.$
  In particular, if $d$ is the unique element in $\mathcal{D} \cap \mathfrak{C}$ guaranteed by {\bf P13} then we also have $$\sum_{x \in \mathfrak{C}^{-1} \cap \mathfrak{C'}} a_x t_d t_x = \sum_{y \in \mathfrak{C}^{-1} \cap \mathfrak{C'}} \pm a_x t_x =0,$$ where the first  equality follows from {\bf P2, P5, P7,} and {\bf P13}.  But this means that $a_x=0$ for all relevant $x$, or in other words, that the $\phi_x$ are linearly independent.  We can therefore conclude that
$ \textup{dim Hom}_{J_\mathbb{C}}(J^\mathfrak{C}_\mathbb{C}, J^\mathfrak{C'}_\mathbb{C}) \geq | \mathfrak{C}^{-1} \cap \mathfrak{C'} |.$
Since this inequality is true for all pairs of left cells $\mathfrak{C}$ and $\mathfrak{C'}$ in $W$, we have
$$
\sum_{\mathfrak{C},\mathfrak{C'}} \textup{ dim Hom}_{J_\mathbb{C}} (J^\mathfrak{C}_\mathbb{C}, J^\mathfrak{C'}_\mathbb{C})  \geq \sum_{\mathfrak{C},\mathfrak{C'}} | \mathfrak{C}^{-1} \cap \mathfrak{C'} |.
$$
The right side of this inequality is just the order of $W$ since each of its elements lies in a unique left and a unique right cell.  On the other hand, by the correspondence resulting from Theorem \ref{theorem:lusztig2} the left side is $$
 \textup{ dim Hom}_{J_\mathbb{C}} \Big( \sum_\mathfrak{C} J^\mathfrak{C}_\mathbb{C}, \sum_\mathfrak{C'}J^\mathfrak{C'}_\mathbb{C}\Big) \\
 = \textup{ dim Hom}_W (\textup{Reg}_W,\textup{Reg}_W)\\  = |W|.
$$
Hence the original inequality must be in fact an equality and the lemma follows.
\end{proof}

We immediately obtain the following corollary, whose proof is identical to that of \cite{lusztig:characters}(12.17).

\begin{corollary}\label{corollary:involutions}  Assume that conjectures {\bf P1-P15} hold and that the left cell modules of $W$ with respect to a weight function $L$ are multiplicity-free.  Then $\mathfrak{C} \cap \mathfrak{C}^{-1}$ is the set of involutions in $\mathfrak{C}$.
\end{corollary}


\section{Type $B_n$}\label{section:typeB}

The goal of this section is to detail the combinatorics of arbitrary rank standard domino tableaux necessary to describe Kazhdan-Lusztig cells in type $B_n$.

\subsection{Domino Tableaux}

Consider a partition $p$ of a natural number $n$.  We will view it as a Young diagram $Y_p$, a left-justified array of squares whose row lengths decrease weakly.  The square in row $i$ and column $j$ of $Y_p$ will be denoted $s_{ij}$ and a pair of squares in $Y_p$ of the form $\{s_{ij},s_{i+1,j}\}$ or $\{s_{ij},s_{i,j+1}\}$ will be called a {\it domino}.  A domino is {\it removable} from $Y_p$ if deleting its underlying squares leaves either another Young diagram containing the square $s_{11}$ or the empty set.

Successive deletions of removable dominos from a Young diagram $Y_p$ must eventually terminate in a staircase partition containing $\binom{r+1}{2}$ squares for some non-negative integer $r$.  This number is determined entirely by the underlying partition $p$ and does not depend on the sequence of deletions of removable dominos.  We will write $p \in \mathcal{P}_r$ and say that $p$ is a {\it partition of rank $r$}.  The {\it core of $p$} is its underlying staircase partition.

\begin{example}  The partition $p=[4,3^2,1]$ lies in the set $\mathcal{P}_2$.   Below are its Young diagram $Y_p$ and a domino tiling resulting from a sequence of deletions of removable dominos exhibiting the underlying staircase partition.
$$
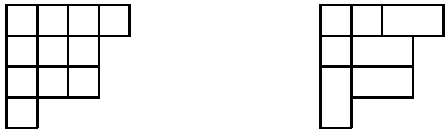
\begin{small}
\begin{tableau}
:.{}.{}.{}.{}\\
:.{}.{}.{}\\
:.{}.{}.{}\\
:.{}\\
\end{tableau}
\end{small}
\hspace{1in}
\begin{small}
\begin{tableau}
:.{}.{}>{}\\
:.{}>{}\\
:^{}>{}\\
:;\\
\end{tableau}
\end{small}
$$
\end{example}

Consider $p \in \mathcal{P}_r$.   It is a partition of the integer $2n+\binom{r+1}{2}$ for some $n$.  A {\it standard domino tableau of rank $r$ and shape $p$} is a tiling of the non-core squares of $Y_p$ by dominos, each of which is labeled by a unique integer from $\{1, \ldots , n\}$ in such a way that the labels increase along its rows and columns.  We will write $SDT_r(p)$ for the set of standard domino tableaux of rank $r$ of shape $p$ and $SDT_r(n)$ for the set of standard domino tableaux of rank $r$ which contain exactly $n$ dominos.

For $T \in SDT_r(n)$, we will say that the square $s_{ij}$ is {\it variable} if $i+j \equiv r \mod 2$ and {\it fixed} otherwise.  As discussed in \cite{garfinkle1} and \cite{pietraho:rscore}, a choice of fixed squares on a tableau $T$ allows us to define two notions, a partition of its dominos into cycles and the operation of moving through a cycle.  The moving through map, when applied to a cycle $c$ in a tableau $T$  yields another standard domino tableau $MT(T,c)$ which differs from $T$ only in the labels of the variable squares of $c$.  If $c$ contains $D(l,T)$, the domino in $T$ with label $l$, then $MT(T,c)$ is in some sense the minimally-affected standard domino tableau in which the label of the variable square in $D(l,T)$ is changed.  We refer the reader to \cite{pietraho:rscore} for the detailed definitions.

If the shape of $MT(T,c)$ is the same as the shape of $T$, we will say that $c$ is a {\it closed cycle}.  Otherwise,  one square will be removed from $T$ (or added to its core) and one will be added.  In this case, we will say the $c$ is {\it open} and denote the aforementioned squares as $s_b(c)$ and $s_f(c),$ respectively.  Finally, if $s_b(c)$ is adjacent to the core of $T$, we will say that $c$ is a {\it core open cycle}.  We will write $OC(T)$ for the set of all open cycles of $T$ and $OC^*(T)$ the subset of non-core open cycles.

\subsection{Generalized Robinson-Schensted Algorithms}

The Weyl group $W_n$ of type $B_n$ consists of the set of signed permutations on $n$ letters, which we write in one-line notation as $w=(w_1 \, w_2 \, \ldots w_n)$.  For each non-negative integer $r$, there is an injective map $$G_r: W_n \rightarrow SDT_r(n) \times SDT_r(n)$$ which is onto the subset of domino tableaux of the same-shape, see \cite{garfinkle1} and \cite{vanleeuwen:rank}. We will write $G_r(x)=(S_r(x),T_r(x))$ for the image of a permutation $x$ and refer to the two components as the {\it left} and {\it right tableaux of $x$}.

\begin{definition}\label{definition:combinatorialcells}
Consider $x, y \in W_n$ and fix a non-negative integer $r$.  We will say that
    \begin{enumerate}
        \item $x \approx^\iota_\mathcal{L} y$ if $T_r(y) = T_r(x)$, and
        \item $x \approx_\mathcal{L} y$ if $T_r(y)= MT(T_r(x),C)$ for some $C\subset OC^*(T_r(x)).$
    \end{enumerate}
\end{definition}

We will call the equivalence classes defined by $\approx^\iota_\mathcal{L}$ {\it irreducible combinatorial left cells of rank $r$} in $W$, and those defined by $\approx_\mathcal{L}$ its {\it reducible combinatorial left cells of rank $r$.}  In the irreducible case, we will say that the combinatorial left cell is {\it represented by the tableau} $T_r(x)$. In the reducible case, we will say that the combinatorial left cell is {\it represented by the set} $\{MT(T_r(x),C) \; | \; C\subset OC^*(T_r(x))\}$  of standard domino tableaux.

\subsection{Cells in type $B_n$}
\label{section:cellsB}

Consider the generators of $W_n$ as in the following diagram:
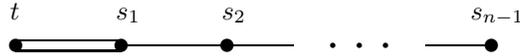
\begin{center}
\begin{picture}(300,30)
\put( 50 ,10){\circle*{5}} \put( 50, 8){\line(1,0){40}} \put( 50,
12){\line(1,0){40}} \put( 48, 20){$t$} \put( 90 ,10){\circle*{5}}
\put(130 ,10){\circle*{5}} \put(230 ,10){\circle*{5}} \put( 90,
10){\line(1,0){40}} \put(130, 10){\line(1,0){25}} \put(170,
10){\circle*{2}} \put(180, 10){\circle*{2}} \put(190,
10){\circle*{2}} \put(205, 10){\line(1,0){25}} \put( 88 ,20){$s_1$}
\put(128, 20){$s_2$} \put(222, 20){$s_{n-1}$}
\end{picture}
\end{center}
Define the weight function $L$ by $L(t)=b$ and
$L(s_i)=a$ for all $i$ and set $s=\frac{b}{a}$.  The following is a
conjecture of Bonnaf\'e, Geck, Iancu, and Lam, and appears as
Conjectures A, B, and D in \cite{bgil}:

\begin{conjecture}\label{conjecture:bgil} Consider a Weyl group of type $B_n$
with a weight function $L$ and parameter $s$ defined as above.
    \begin{enumerate}
        \item When $s \not\in \mathbb{N}$, the Kazhdan-Lusztig left cells coincide
                with the irreducible combinatorial left cells of rank $\lfloor s \rfloor$.
        \item  When $s \in \mathbb{N}$, the Kazhdan-Lusztig left cells coincide
                with the reducible combinatorial left cells of rank $s-1$.
    \end{enumerate}
\end{conjecture}

This conjecture is well-known to be true for $s=1$ by work of
Garfinkle \cite{garfinkle3}, and has been verified when $s>n-1$ by
Bonnaf\'e and Iancu \cite{bonnafe:iancu}. It has also been shown to hold for all
values of $s$ when $n \leq 6$, see \cite{bgil}. Furthermore, assuming {\bf P1-P15}, C.~Bonnaf\'e has shown the conjecture to be true in the irreducible case, and that in the reducible case, Kazhdan-Lusztig left cells are unions of the reducible combinatorial left cells \cite{bonnafe:knuth}.

\section{Constructible Representations in Type $B_n$}\label{section:constructible}
M.~Geck has shown that if Lusztig's conjectures {\bf P1-P15} hold, then the $W$-modules carried by the Kazhdan-Lusztig left cells of an unequal parameter Hecke algebra are precisely the constructible ones \cite{geck:constructible}. Defined in the unequal parameter setting by Lusztig in \cite{lusztig:unequal}(20.15),  constructible modules arise via truncated induction and tensoring with the sign representation.  The goal of this section is to give a combinatorial description of the effects of these two operations on $W$-modules in type $B_n.$  Our approach is based on the equal-parameter results of \cite{mcgovern:leftcells}.

\subsection{Irreducible $W_n$-modules}\label{section:irreducible}

Let us restrict our attention to type $B_n$, write $W_n$ for the corresponding Weyl group, and define constants $a$, $b$, and $s,$ as in Section \ref{section:cellsB}.   We begin by recalling the standard parametrization of irreducible $W_n$-modules.  Let $\mathcal{P}^2$ be the set of ordered pairs of partitions and $\mathcal{P}^2(n)$ be the subset of $\mathcal{P}^2$ where the combined sum of the parts of both partitions is $n$.
\begin{theorem} \label{theorem:irrW}
The set of irreducible representations of $W_n$ is parametrized by $\mathcal{P}^2(n)$.
If we write $[(d,f)]$ for the representation corresponding to $(d,f) \in \mathcal{P}^2(n)$, then $$[(f^t,d^t)] \cong [(d,f)] \otimes \textup{sgn},$$ where $p^t$ denotes the transpose of the partition $p$.
\end{theorem}

In this form, the connection between irreducible $W_n$-modules and the description of left cells in Conjecture \ref{conjecture:bgil} is not clear.  To remedy this, we would like to restate Theorem \ref{theorem:irrW} in terms of partitions of arbitrary rank which arise as shapes of the standard domino tableaux in this conjecture.  Thus let $r= \lfloor s \rfloor$ if $s \not\in \mathbb{N}$, $r= s-1$ otherwise, and write  $\epsilon=s - \lfloor s \rfloor$.
As an intermediary to this goal, we define the notion of a {\it symbol of defect $t$ and residue $\epsilon$} for a non-negative integer $t$ and $0\leq \epsilon < 1$.  It is an array of non-negative numbers of the form
$$
\Lambda= \left(
\begin{array}{cccccc}
  \lambda_1 + \epsilon & \lambda_2 + \epsilon &  &\ldots  & & \lambda_{N+t} + \epsilon \\
  {} & \mu_1 & \mu_2 & \ldots  & \mu_N & {}
\end{array}
\right)
$$
where the (possibly empty) sequences $\{\lambda_i\}$ and $\{\mu_i\}$ consist of integers and are strictly increasing. If we define a related symbol by letting
$$
\Lambda'= \left(
\begin{array}{ccccccc}
  \epsilon & \lambda_1+1+ \epsilon  & \lambda_2+2+ \epsilon &\ldots  & & \lambda_{N+t} +N+t+ \epsilon\\
   {} & 0 & \mu_1+1 & \ldots  & \mu_N+N & {}
\end{array}
\right)
$$
then the binary relation defined by setting $\Lambda \sim \Lambda'$ generates an equivalence relation.  We will write $Sym^\epsilon_t$ for the set of its equivalence classes.

We describe two maps between symbols and partitions.  A partition can be used to construct a symbol in the following way.
If $p = (p_1,p_2, \ldots, p_k)$, form  $p^\sharp=(p_1, p_2,
\ldots, p_{k'})$ by adding an additional zero term to $p$ if the
rank of $p$ has the same parity as $k$.  Dividing the set
$\{p_i+k'-i\}_{i=1}^{k'}$ into its odd and even parts yields two sequences
$$\{2\mu_i+1\}_{i=1}^N\text{ and }\{2\lambda_i\}_{i=1}^{N+t}$$ for some non-negative integer $t$.  A
symbol $\Lambda_p$ of defect $t$ and residue $\epsilon$ corresponding to $p$ can now be defined by
arranging the integers $\lambda_i$ and $\mu_i$ into an array as above.

Given a symbol of defect
$t$ and residue $\epsilon$, it is also possible to construct an ordered pair of partitions.
With $\Lambda$ as above, let $$d_\Lambda = \{\lambda_i - i
+1\}_{i=1}^{N+t}\text{ and }f_\Lambda =\{\mu_i - i +1\}_{i=1}^N.$$

Both constructions are well-behaved with respect to the equivalence on symbols defined above.
The next theorem follows from \cite{james:kerber}(2.7).

\begin{theorem}
The maps $p \mapsto \Lambda_p$ and  $\Lambda \mapsto (d_{\Lambda},
f_{\Lambda})$ define  bijections
$$\mathcal{P}_r \rightarrow Sym^\epsilon_{r+1} \rightarrow
\mathcal{P}^2$$ for all values of $r$ and $\epsilon$.  Consequently, their composition
yields a bijection between $\mathcal{P}_r(n)$ and
$\mathcal{P}^2(n)$. \label{theorem:bijections}
\end{theorem}

This result allows us to custom tailor a parametrization of irreducible $W_n$-modules to each value of the parameter $s$ by defining $r$ and $\epsilon$ as above.  Together with Lusztig's Lemma 22.18 of \cite{lusztig:unequal}, the present theorem implies the following alternate parametrization of the representations of $W_n$ in terms of symbols.  A parametrization in terms of partitions of rank $r$ follows.
\begin{corollary}
If we fix  values of the defect $r$ and residue $\epsilon$, then
the set of irreducible representations of $W_n$ is parametrized by the set of equivalence classes of symbols
$\{ \Lambda \in Sym^\epsilon_{r+1} \; | \; \text{ parts of $d_\Lambda$ and $f_\Lambda$ sum to $n$}\}.$
Writing $[\Lambda]$ for the representation corresponding to $\Lambda$, we have
$$[\bar{\Lambda}]=[\Lambda] \otimes \textup{sgn}$$
where the symbol $\bar{\Lambda}$ is defined from $\Lambda$ by the following procedure.  Write $\Lambda$ as above and let $\tau$ be the integer part its largest entry.  Then the integer parts of the top and bottom rows of $\bar{\Lambda}$ consist of the complements of $\{\tau-\mu_i\}_i$ and $\{\tau-\lambda_i\}_i$ in $[0,\tau] \cap \mathbb{Z}$, respectively.
\end{corollary}

\begin{corollary} \label{corollary:tensoringwithsign}
If we fix a non-negative integer $r$, then the set of irreducible representations of $W_n$ is parametrized by $\mathcal{P}_r(n)$.
Writing  $[p]$ for the representation corresponding to $p \in \mathcal{P}_r(n)$, we have $$[p^t] \cong [p] \otimes \textup{sgn},$$ where $p^t$ is the transpose of the partition $p$.
\end{corollary}

\begin{example} Let $s=3 \frac{1}{2}$, so that $r=3$ and $\epsilon = \frac{1}{2}$, and  consider the irreducible representation $[((1^3), (1))]$ of $W_4$.  Then according to the above parametrizations,
$[((1^3), (1))]=[(4,3,2^2)]=[\Lambda_{(4,3,2^2)}]$
where
$$
\Lambda_{[(4,3,2^2)]}= \left(
\begin{array}{cccccc}
  \frac{1}{2}  & 2 \frac{1}{2} &  3 \frac{1}{2} & 4 \frac{1}{2}\\
  {} & {} & \hspace{-.3in}1 & {}
\end{array}
\right)
$$
is a symbol of defect 3 and residue $\frac{1}{2}$.  Note that $((1^3), (1)) \in \mathcal{P}^2(4)$, $(4,3,2^2) \in \mathcal{P}_2(4)$, and $\Lambda_{(4,3,2^2)}$ is a representative of a class in $Sym^{\epsilon}_3$ for $\epsilon = 1/2$.  Furthermore,
$[((1^3), (1))] \otimes \textup{sgn}  = [((1),(3))] =  [(4,3,2^2)] \otimes \textup{sgn} = [(4^2,2,1)]  =
 [\Lambda_{(4,3,2^2)}]\otimes \textup{sgn} = [\Lambda_{(4^2,2,1)}],$ where
$$
\Lambda_{[(4^2,2,1)]}= \left(
\begin{array}{cccccc}
  \frac{1}{2}  & 1 \frac{1}{2} &  2 \frac{1}{2} & 4 \frac{1}{2}\\
  {} & {} & \hspace{-.3in}3 & {}
\end{array}
\right).
$$
\end{example}

We will need the following lemma, which holds for finite $W$ whenever {\bf P1-P15} hold.  It is a combination of \cite{lusztig:unequal}(11.7) and \cite{lusztig:unequal}(21.5).

\begin{lemma} \label{lemma:longword} Consider a Kazhdan-Lusztig left cell $\mathfrak{C} \subset W$ and let $w_0$ be the longest element of $W$.  Then $\mathfrak{C}w_0$ is also a left cell in $W$, and $[\mathfrak{C}w_0] \cong \mathfrak{C} \otimes \textup{sgn}$ as $W$-modules.
\end{lemma}

\subsection{Truncated Induction}

We now turn to a combinatorial description of truncated induction in terms of the above parameter sets.
If $\pi$ is a representation of $W_I$, a parabolic subgroup of $W_n$, Lusztig defined a representation
$J_{W_I}^W (\pi)$ of $W=W_n$, \cite{lusztig:unequal}(20.15).  Its precise definition depends of the parameters of the underlying Hecke algebra, so it is natural to expect that this is manifested in the combinatorics studied above.  Following \cite[\S 2]{mcgovern:leftcells} and \cite{lusztig:classofirreducible}, we note that due to the transitivity of truncated induction and the fact that the situation in type $A$ is well-understood, we need to only understand how truncated induction works when $W_I$ is a maximal parabolic subgroup whose type $A$ component acts by the sign representation on $\pi$.  Henceforth,
let $W_I$ be a maximal parabolic subgroup in $W_n$ with factors  $W'$ of type $B_m$ and $S_l$ of type $A_{l-1}$, where $m+l = n$; furthermore, write $\textup{sgn}_l$ for the sign representation of $S_l$.

Truncated induction behaves well with respect to cell structure. In fact, the following lemma holds for general $W$.
\begin{lemma}[\cite{geck:leftcells}]\label{lemma:paraboliccells} Let $\mathfrak{C}'$ be a left cell of $W_I$.  Then we have
$$J_{W_I}^W([\mathfrak{C}']) \cong [\mathfrak{C}],$$
where $\mathfrak{C}$ is the left cell of $W$ such that $\mathfrak{C}'\subset \mathfrak{C}.$
\end{lemma}

We first provide a description of the situation in type $B_n$ in terms of symbols.  Consider a symbol $\Lambda'$ of defect $r+1$ and residue $\epsilon$; via the equivalence on symbols, we can assume that it has at least $l$ entries.  If the set of $l$ largest entries of $\Lambda'$ is uniquely defined, then let $\Lambda$ be the symbol obtained by increasing each of the entries in this set by one.  If it is not, then let $\Lambda^\textup{I} $ and $\Lambda^{\textup{II}}$ be the two symbols obtained by increasing the largest $l-1$ entries of $\Lambda'$ and then each of the two $l$th largest entries in turn by one.

\begin{proposition}[\cite{lusztig:unequal}(22.17)] The representation
$J_{W_I}^W ( [\Lambda'] \otimes \textup{sgn}_l )$ is $[\Lambda]$ if the set of $l$ largest entries of $\Lambda'$ is uniquely defined, and $[\Lambda^
\textup{I}] + [\Lambda^{\textup{II}}]$  if it is not.  The former is always the case if $[\Lambda']$ is a symbol of residue $\epsilon \neq 0$.
\end{proposition}

It is not difficult to reformulate this result in terms of partitions of rank $r$.  Consider a partition $p=(p_1,p_2, \ldots p_k) \in \mathcal{P}_r$.  We can assume that $k \geq l$ by adding zero parts to $p$ as necessary.   Let $k'$ be the number of parts of $p^\sharp$.  Define
\begin{align*}
p^\textup{I} & = (p_1+2, \ldots , p_l+2, p_{l+1}, \ldots, p_k), \text{ and} \\
p^\textup{II} & = (p_1+2, \ldots , p_{l-1}+2 , p_l+1, p_{l+1}+1, p_{r+2}, \ldots, p_k).
\end{align*}
Note that both $p^\textup{I}$ and $p^\textup{II}$ are again partitions of rank $r$.

\begin{corollary} \label{corollary:truncated}
The representation
$J_{W_I}^W ([p] \otimes \textup{sgn}_l)$ produced by truncated induction is $[p^\textup{I}]$ whenever $p_l > p_{l+1}$, $p_l+r-l$ is odd, or $\epsilon \neq 0$.  Otherwise, $$J_{W_I}^W ( [p] \otimes \textup{sgn}_l) = [p^\textup{I}]+[p^\textup{II}].$$
\end{corollary}

\begin{proof}Using the results of the preceding proposition, we have to check under what conditions the set of $l$ largest entries in a symbol $\Lambda'$ is uniquely defined and then determine the preimages of the symbols $\Lambda^\textup{I}$ and $\Lambda^{\textup{II}}$ under the map of Theorem \ref{theorem:bijections}.
When $\epsilon \neq 0$, the $l$ largest entries in $\Lambda'$ are uniquely determined since all of its entries must be distinct. When $\epsilon =0$, there will be an ambiguity in determining the $l$ largest entries iff $p_l+k'-l$ and $p_{l+1}+k'-l-1$ are consecutive integers with the first one being odd.  Together with the observation that $k'$ is always of the opposite parity from $r$, this gives us the conditions of the proposition.  Determining the partitions corresponding to $\Lambda^\textup{I}$ and $\Lambda^\textup{II}$ is then just a simple calculation.
\end{proof}

Note that the parity conditions of the proposition imply that in the case when $J_{W_I}^W ([p] \otimes \textup{sgn}_l)$ is reducible, the square $s_{l,p_l+1}$ of the Young diagrams of $p^\textup{I}$ and $p^\textup{II}$ is fixed.  In particular, this means that when endowed with the maximal label, the domino $\{s_{l,p_l+1}, s_{l,p_l+2}\}$ constitutes an open cycle in a domino tableau of shape $p^\textup{I}$.  Its image under the moving through map is $\{s_{l+1,p_l+1}, s_{l,p_l+1}\}$ with underlying partition $p^\textup{II}.$  This observation leads to the following lemma:

\begin{lemma}\label{lemma:inductionshapes}
Let $n=m+l$ and consider  $w'=(w_1 \, w_2 \ldots w_m) \in W_m$.  Write $T'=T_r(w')$ for its right tableau of rank $r$ and define a set of partitions $$ \mathbb{P}' = \{shape \, MT(T', C) \; | \; C \subset OC^*(T')\} \subset \mathcal{P}_r(m).$$
Define the set $\mathbb{P} = \{ p^\textup{I} \; | \; p \in \mathbb{P}'\} \cup  \{ p^\textup{II} \; | \; p \in \mathbb{P}' \text{ and $p_l = p_{l+1}$ with $p_l+r-l$ even} \}.$
If $w = (w_1 \, w_2 \ldots w_m \; n \; n-1 \, \ldots m+1) \in W_n$ with right tableau $T=T_r(w)$, then
$$\mathbb{P}=\{shape \, MT(T, C) \; | \; C \subset OC^*(T)\} \subset \mathcal{P}_r(n).$$
\end{lemma}

\begin{proof}  The lemma relates the non-core open cycles in $T'$ to the non-core open cycles in $T$, hence it follows from the description of the behavior of cycles under domino insertion in \cite{pietraho:rscore}(3.6).  However, things are really simpler than that, and we describe the situation fully.  Note that $T$ is obtained from $T'$ by placing horizontal dominos with labels $m+1$ through $n$ at the end of its first $l$ rows.  Essentially, there are four possibilities.  We write $s_{ij}$ for the left square of the domino added to row $i$ and let $p= shape \, T'$.
    \begin{enumerate}
        \item $s_{ij}=S_f(c)$ for a cycle $c$ of $T'$.  Then the domino joins the cycle $c$ and the final square of the new cycle is $s_{i,j+2}$.
        \item $s_{ij-1}=S_b(c)$ for a cycle $c$ of $T'$.  Then the domino joins the cycle $c$ and the beginning square of the new cycle is $s_{i,j+1}$.
        \item $p_{i-1} = p_{i}$ with $p_i+r-i$ odd.  Then the dominos with labels $m+i-1$ and $m+i$ in $T$ form a closed cycle in $T$.
        \item \label{case:extracycle} $p_{l} = p_{l+1}$ with $p_l+r-l$ even.  Then the domino with label $n$ forms a singleton non-core open cycle in $T$ which does not correspond to a cycle in $T'$.
    \end{enumerate}
If $C \subset OC^*(T')$ and $\widetilde{C}$ is the set of the corresponding cycles in $T$, then it is clear from the above description that  $\{shape \, MT(T, \widetilde{C}) \; | \; C \subset OC^*(T')\} =\{ p^\textup{I} \; | \; p \in \mathbb{P}'\}$.  If case (\ref{case:extracycle}) arises and $T$ has an additional non-core open cycle $c =\{n\}$, then $\{shape \, MT(T, \widetilde{C}\cup c) \; | \; C \subset OC^*(T')\} =\{ p^\textup{II} \; | \; p \in \mathbb{P}'\}.$  The lemma follows.
\end{proof}

\begin{example}
Let $s=3$, so that $r=2$ and $\epsilon=0$, and
consider the partition $(4,3,2^3) \in \mathcal{P}_2(5)$.  It corresponds to the symbol
$$
\Lambda_{[(4,3,2^3)]}= \left(
\begin{array}{cccccc}
  1 & 2 &  3  & 4 \\
  {} & {} & \hspace{-.2in}1 & {}
\end{array}
\right)
\in Sym_3^0
$$
For $l=4$, we have
$J_{W_I}^W ([(4,3,2^3)]\otimes \textup{sgn}_4) = [(6,5,4,3^2)] + [(6,5,4^2,2)].$  Note that both partitions lie in $\mathcal{P}_2(9)$.  In terms of symbols,
$$J_{W_I}^W ([\Lambda_{(4,3,2^3)}] \otimes \textup{sgn}_4)= \left[
\left(
\begin{array}{cccccc}
  2 & 3 &  4  & 5 \\
  {} & {} & \hspace{-.2in}1 & {}
\end{array}
\right)
\right] +
\left[
\left(
\begin{array}{cccccc}
  1 & 3 &  4  & 5 \\
  {} & {} & \hspace{-.2in}2 & {}
\end{array}
\right)
\right]
$$

\end{example}

\section{$W_n$-module structure and standard domino tableaux}\label{section:modulestructure}

Viewing cells as constructible representations allows us to examine their structure inductively.  Using the description of truncated induction and tensoring with sign  derived in the previous section we describe the $W_n$-module carried by each cell in terms of the parametrization of irreducible $W_n$-modules of Section \ref{section:irreducible}.  We begin with a few facts about combinatorial cells.

\begin{lemma} \label{lemma:intersection}
Consider two combinatorial left cells $\mathfrak{C}$ and $\mathfrak{C}'$ in $W_n$ of rank $r$ represented by sets $\mathbb{T}$ and $\mathbb{T}'$ of rank $r$ standard domino tableaux.  Then
$$| \mathfrak{C} \cap  \mathfrak{C}'^{-1}| = M$$ where $M$ is the number of tableaux in $\mathbb{T}$ whose shape matches the shape of a tableau in  $\mathbb{T}'$.
\end{lemma}

\begin{proof}  Suppose first that $\mathfrak{C}$ and $\mathfrak{C}'$ are irreducible so that $\mathbb{T}=\{T\}$ and $\mathbb{T}'=\{T'\}$.  If they are of the same shape, then the intersection $ \mathfrak{C} \cap  \mathfrak{C}'^{-1} = G_r^{-1}(T',T)$; otherwise, it is empty.

On the other hand,  if $\mathfrak{C}$ and $\mathfrak{C}'$ are reducible,  then let $J$ consist of  the tableaux in $\mathbb{T}$ whose shape matches the shape of a tableau in $\mathbb{T}'$ and define $|J|=M$.  Recall that by the definition of a combinatorial left cell, $\mathbb{T} = \{MT(T,C) \; | \; C \in OC^*{T}\}$ for some tableau $T$ and therefore $\mathbb{T}$ consists of only tableaux of differing shapes.
If $T \in J$, write $T'$ for the the unique tableau in $\mathbb{T}'$ of the same shape as $T$.  Then
$$ \mathfrak{C} \cap  \mathfrak{C}'^{-1} = \bigcup_{T \in J} G_r^{-1}(T',T).$$

\end{proof}

We can obtain a slightly better description of the intersection of a combinatorial left cell and a combinatorial right cell by recalling the definition of an extended open cycle in a tableau relative to another tableau of the same shape.  See \cite{garfinkle2}(2.3.1) or \cite{pietraho:equivalence}(2.4) for the details.  In general, an extended open cycle is a union of open cycles.

\begin{corollary}\label{corollary:intersection}
Consider two reducible combinatorial left cells $\mathfrak{C}$ and $\mathfrak{C}'$ in $W_n$ of rank $r$ represented by sets $\mathbb{T}$ and $\mathbb{T}'$ of rank $r$ standard domino tableaux.  If $T \in \mathbb{T}$ and $T' \in \mathbb{T}'$ are of the same shape and $m$ is the number of non-core extended open cycles $m$ in $T$ relative to $T'$, then
$$|\mathfrak{C} \cap  \mathfrak{C}'^{-1}| = 2^m.$$
\end{corollary}

\begin{proof}
An extended open cycle in $T$ relative to $T'$ is a minimal set of open cycles in $T$  and $T'$ such that moving through it produces another pair of tableaux of the same shape.  Consequently, moving through two different extended open cycles are independent operations.  Noting that
$$\mathbb{T} =\{ MT(T,C) \; | \; C\subset OC^*(T)\} \text{ and } \mathbb{T}' =\{ MT(T',C) \; | \; C\subset OC^*(T')\},$$
we have that a tableau-pair $(S,S') \in \mathbb{T} \times \mathbb{T}'$ is same-shape iff it differs from $(T,T')$ by moving through a set of non-core extended open cycles in $T$ relative to $T'$.
Thus, if $E$ is the set of non-core extended open cycles in $T$ relative to $T'$, then
$$\mathfrak{C} \cap  \mathfrak{C}'^{-1}=\bigcup_{D \subset E} G_r^{-1}\big(MT((T',T),D)\big),$$
from which the result follows.
\end{proof}

Recall the parameter $s$ derived from a weight function $L$ in type $B_n$.  We will call a Kazhdan-Lusztig left cell in this setting a {\it left cell of weight $s$}.   C.~Bonnaf\'e \cite{bonnafe:knuth} has shown that:
    \begin{itemize}
        \item under the assumption that statements {\bf P1-P15} of Section \ref{section:klcells} hold,
                when $s \not \in \mathbb{N}$, left cells of weight $s$ are precisely the irreducible combinatorial left cells of rank $r=\lfloor s \rfloor$, and
        \item when $s \in \mathbb{N}$, left cells of weight $s$ are unions of reducible combinatorial left cells of rank $r=s-1$.
    \end{itemize}

In this way, as in Definition \ref{definition:combinatorialcells}, we can say that a left cell of weight $s$ is {\it represented by} a set of standard domino tableaux of rank $r$.  For non-integer $s$, this set consists of the unique tableau representing the irreducible combinatorial left cell, and in the latter, it is the union of the sets of tableaux representing each of the reducible combinatorial cells in the Kazhdan-Lusztig cell.  In what follows, we assume that statements {\bf P1-P15} hold.

\begin{lemma}\label{lemma:disjointshapes}
Suppose that $\mathfrak{C}$ is a left cell of weight $s$ and
$\mathfrak{C} = \coprod_i \mathfrak{D}_i $
is its decomposition into combinatorial left cells of rank $r$.  If we let $\mathbb{T}_i$ be the set of domino tableaux representing $\mathfrak{D}_i,$ then the set of shapes of tableaux in $\mathbb{T}_i$ is disjoint from the set of shapes of tableaux in $\mathbb{T}_j$ whenever $i \neq j$.
\end{lemma}

\begin{proof}
By Corollary \ref{corollary:involutions}, $\mathfrak{C}\cap \mathfrak{C}^{-1}$ consists of the involutions in $\mathfrak{C}$.  The set of involutions in each combinatorial cell $\mathfrak{D}_i$ consists of $\mathfrak{D}_i \cap \mathfrak{D}_i^{-1}$.  This forces $\mathfrak{D}_i \cap \mathfrak{D}_j^{-1}=\varnothing$ whenever $i\neq j$, which can only occur if  the set of shapes of tableaux in $\mathbb{T}_i$ is disjoint from the set of shapes of tableaux in $\mathbb{T}_j$, by Lemma \ref{lemma:intersection}.
\end{proof}

We first show that the shapes of the standard domino tableaux of rank $r$ representing a left cell of weight $s$  determine its $W_n$-module structure:

\begin{definition} Suppose $\mathbb{T}$ is a set of standard domino tableaux of rank $r$.  For $T \in \mathbb{T}$, we will write $p_T \in \mathcal{P}_r(n)$ for its underlying partition, and define
$$[\mathbb{T}] = \bigoplus_{T \in \mathbb{T}} [p_T].$$
\end{definition}

\begin{lemma}\label{lemma:shapes}
Suppose that $\mathfrak{C}$  and  $\mathfrak{C'}$ are left cells of weight $s$ in $W_n$ and
$$ \mathfrak{C} = \coprod_{i \leq c} \mathfrak{D}_i \text{ as well as } \mathfrak{C}' = \coprod_{i \leq d} \mathfrak{D}'_i$$
are their decompositions into combinatorial left cells of rank $r$.  Suppose that each $\mathfrak{D}_i$ and $\mathfrak{D}'_i$ is represented by the set of rank $r$ tableaux $\mathbb{T}_i$ and $\mathbb{T}'_i$, respectively.  Then $[\mathfrak{C}] \cong [\mathfrak{C'}]$ iff $c=d$ and, suitably ordered, $[\mathbb{T}_i] \cong [\mathbb{T}'_i]$ for all $i$.
\end{lemma}

\begin{proof}
For clarity, we treat the integer and non-integer values of $s$ separately.  First assume $s \not \in \mathbb{N}$ so that $c=d=1$ and take  $\{T\}= \mathbb{T}_1$ and $\{T'\}=\mathbb{T}_1'.$  By Lemmas \ref{lemma:12.15} and \ref{lemma:intersection}, we have
$\dim \textup{Hom}_W([\mathfrak{C}],[\mathfrak{C}]) = \dim \textup{Hom}_W([\mathfrak{C}'],[\mathfrak{C}'])=1.$  Furthermore, we have that
$\dim \textup{Hom}_W([\mathfrak{C}],[\mathfrak{C}']) = | \mathfrak{C} \cap \mathfrak{C}'^{-1} |=1$
if and only if the shapes of $T$ and $T'$ coincide; otherwise, $\dim \textup{Hom}_W([\mathfrak{C}],[\mathfrak{C}'])=0$.  The lemma follows.

Next, assume $s \in \mathbb{N}$. Suppose first that $[\mathfrak{C}] \cong [\mathfrak{C'}].$  Then
$\dim \textup{Hom}(\mathfrak{C},\mathfrak{C}) = \dim \textup{Hom}(\mathfrak{C}',\mathfrak{C'}) =\dim \textup{Hom}(\mathfrak{C},\mathfrak{C}')$, and by Lemma \ref{lemma:12.15}, $|\mathfrak{C} \cap \mathfrak{C}^{-1}|= |\mathfrak{C}' \cap \mathfrak{C'}^{-1}|=|\mathfrak{C} \cap \mathfrak{C'}^{-1}|.$  By Lemma \ref{lemma:disjointshapes}, we have
$$\sum_{i\leq c} |\mathfrak{D}_i \cap \mathfrak{D}_i^{-1}| = \sum_{i\leq d} |\mathfrak{D}'_i \cap {\mathfrak{D}'_i}^{-1}| = \sum_{i,j} |\mathfrak{D}_i \cap {\mathfrak{D}'_j}^{-1}|.$$
We can now use Corollary \ref{corollary:intersection} to examine the terms of this equality.  For a combinatorial cell $\mathfrak{D}_i$, there is at most one cell $\mathfrak{D}'_{i'}$ such that there are $T_i \in \mathbb{T}_i$ and $T'_{i'} \in \mathbb{T}'_{i'}$ of the same shape, by Lemma \ref{lemma:disjointshapes}. Let $I$ be the set of $i$ for which this occurs. Let $c_i$ and $d_i$ be the numbers of non-core open cycles in $T_i$ and $T'_{i'}$ and for each $i\in I$, let $m_i$ be the number of non-core extended open cycles in $T_i$ relative to $T'_{i'}$.  Then $m_i \leq c_i, d_{i'}$ with equality iff the non-core extended open cycles are just the non-core open cycles.   By Corollary \ref{corollary:intersection},
$\sum_{i\leq c} |\mathfrak{D}_i \cap \mathfrak{D}_i^{-1}| = \sum_{i\leq c} 2^{c_i},$ $ \sum_{i\leq d} |\mathfrak{D}'_i \cap {\mathfrak{D}'_i}^{-1}| = \sum_{i\leq d} 2^{d_i}$, and $\sum_{I} |\mathfrak{D}_i \cap {\mathfrak{D}'_{i'}}^{-1}| = \sum_I 2^{m_i}.$  But the previous equation now implies that $m_i=c_i=d_{i'}$, $c=d$, $I=\{1, \ldots, c\}$ and by the definition of a combinatorial left cell in our setting, that $[\mathbb{T}_i] \cong [\mathbb{T}'_{i'}]$ for all $i \in I$.

Conversely, assume that $c=d$ and $[\mathbb{T}_i] \cong [\mathbb{T}'_i]$ for all $i$ and choose tableaux $T_i \in \mathbb{T}_i$ and $T'_i \in \mathbb{T}'_i$ of the same shape.  By the definition of combinatorial cells, there is a correspondence between the non-core open cycles of $T_i$ and those of $T'_i$ such that their beginning and final squares coincide, implying that the set of non-core extended open cycles in $T_i$ relative to $T'_i$ is precisely the set of open cycles of $T_i$.  Therefore, for each $i$ we have
$|\mathfrak{D}_i \cap \mathfrak{D}_i^{-1}| = |\mathfrak{D}_i \cap {\mathfrak{D}'_i}^{-1}|.$  Consequently, by Lemmas \ref{lemma:disjointshapes} and \ref{lemma:12.15}, and Corollary \ref{corollary:intersection}:
$$\dim \textup{Hom} (\mathfrak{C}, \mathfrak{C}')= \sum_i |\mathfrak{D}_i \cap {\mathfrak{D}'_i}^{-1}| = \sum_i |\mathfrak{D}_i \cap \mathfrak{D}_i^{-1}| = \dim \textup{Hom} (\mathfrak{C}, \mathfrak{C}).$$
Reversing the roles of $\mathfrak{C}$ and $\mathfrak{C}'$ above implies the desired result.
\end{proof}

\begin{theorem} \label{theorem:main} Suppose that $\mathfrak{C}$ is a left cell of weight $s$ in $W_n$  represented by a set $\mathbb{T}$ of standard domino tableaux of rank $r$.  Then $[\mathfrak{C}] \cong [\mathbb{T}]$ as $W_n$-modules.
\end{theorem}

\begin{proof}
In light of the result from Lemma \ref{lemma:shapes}, we can prove the theorem by verifying it holds for a representative of each isomorphism class of left cells.  Under our assumptions, the results of \cite{geck:constructible} hold and left cell modules coincide with constructible representations of $W_n$.  Therefore, a representative of each isomorphism class of left cells can be obtained by repeated truncated induction and tensoring with sign.  Recall our description of irreducible $W_n$-modules by partitions of rank $r$.  Via Corollaries \ref{corollary:tensoringwithsign} and \ref{corollary:truncated}, we have a description of both operations on the level of partitions.  We verify that the effect of truncated induction and tensoring with sign on the shapes of the tableaux representing a left cell is the same, and the theorem follows by induction.

We treat the integer and non-integer values of $s$ separately.  First assume $s \not \in \mathbb{N}$, so that each left cell is represented by a single tableau.  We begin by investigating the effect on tensoring with sign.
If $[\mathfrak{C}]$ is a left cell module and $w \in \mathfrak{C}$, then $\mathfrak{C}$ is represented by the tableau $T_r(w)$ of shape $p$.   By Lemma \ref{lemma:longword}, $\mathfrak{C}w_0$ is also a left cell and  $[\mathfrak{C}w_0] \cong [\mathfrak{C}] \otimes \textup{sgn}.$  It is represented by the tableau $T_r(ww_0) = T_r(w)^t$ of shape $p^t$.  By Corollary \ref{corollary:tensoringwithsign}, if we assume that $[\mathfrak{C}]$ carries the irreducible module associated to the shape of its representative tableau, then so does $[\mathfrak{C}w_0]\cong [\mathfrak{C}] \otimes \textup{sgn}.$

For the case of truncated induction, consider a maximal parabolic subgroup $W_I = W_m \times S_l$ of $W_n$.  Choose $w'=(w_1 \, w_2 \ldots w_m) \in W_m$ and let $\mathfrak{C}'$ be its left cell, represented by the tableau $T'=T_r(w')$.  Let $p = shape \; T'$.  By Lemma \ref{lemma:paraboliccells}, $J_{W_I}^W( [\mathfrak{C}'] \otimes  \textup{sgn}_l )=[\mathfrak{C}]$ for a left cell $\mathfrak{C} \subset W_n$ and furthermore, the element $w=(w_1 \, w_2 \ldots w_m \; n \; n-1 \ldots m+1) \in \mathfrak{C}$.  The left cell $\mathfrak{C}$ is represented by the tableau $T_r(w)$ whose shape is $p^I$, using the notation of (\ref{corollary:truncated}).  By Corollary \ref{corollary:truncated}, if we assume that $[\mathfrak{C}']$ carries the irreducible module associated to the shape of its representative tableau, then so does $[\mathfrak{C}]\cong J_{W_I}^W( [\mathfrak{C}'] \otimes  \textup{sgn}_l ).$

Next assume $s  \in \mathbb{N}$, so that each left cell is represented by a family of rank $r$ standard domino tableaux.  Again, we begin by investigating the effect on tensoring with sign.  Suppose $\mathfrak{C}$ is a left cell  represented by the set $\mathbb{T}$ and for each $T \in \mathbb{T}$, $w_T \in W_n$ is chosen so that $T_r(w_T)=T$.  By Lemma \ref{lemma:longword}, $\mathfrak{C}w_0$ is also a left cell and $[\mathfrak{C}w_0] \cong [\mathfrak{C}] \otimes \textup{sgn}.$   It is represented by the set of tableaux $T_r(w_T w_0) = T_r(w_T)^t$ (for $T \in \mathbb{T}$), which we write as $\mathbb{T}^t$.   By Corollary \ref{corollary:tensoringwithsign}, if we assume that $[\mathfrak{C}]$ carries the module $[\mathbb{T}]$  then $[\mathfrak{C}w_0]\cong [\mathfrak{C}] \otimes \textup{sgn}$ carries the module $[\mathbb{T}^t].$

For the case of truncated induction, again consider a maximal parabolic subgroup $W_I = W_m \times S_l$ of $W_n$.  Let $\mathfrak{C}'$ be a left cell of $W_m$ and  let $\mathfrak{C}'=\coprod_i \mathfrak{D}'_i$ be its decomposition into combinatorial left cells.  Suppose that $\mathfrak{D}'_i$ is represented by the set $\mathbb{T}'_i$ of domino tableaux and let $\mathbb{T}'= \coprod_i \mathbb{T}'_i$.  By definition of combinatorial left cells, every $\mathbb{T}'_i = \{ MT(T'_i, C) \; | \; C \subset OC^*(T'_i)\}$ for some rank $r$ standard domino tableau $T'_i$.  For each $i$, choose $\widetilde{w}^i=(w^i_1 \, w^i_2 \ldots w^i_m) \in W_m$ with $T'_i=T_r(\widetilde{w}^i)$ so that $\widetilde{w}^i \in \mathfrak{D}'_i$.    By Lemma \ref{lemma:paraboliccells}, $J_{W_I}^W( [\mathfrak{C}'] \otimes  \textup{sgn}_l )=[\mathfrak{C}]$ for a left cell $\mathfrak{C} \subset W_n$.  Furthermore, $w^i=(w^i_1 \, w^i_2 \ldots w^i_m \; n \; n-1 \ldots m+1) \in \mathfrak{C}$ and if $T_i = T_r(w^i)$, then $\mathfrak{C}$ is represented by the set of
tableaux $\mathbb{T}= \coprod_i \{MT(T_i, C) \; | \; C \subset OC^*(T_i)\}$. Lemma \ref{lemma:inductionshapes} describes the shapes of the tableaux in $\mathbb{T}$ in terms of the shapes of the tableaux in $\mathbb{T}'$.  This, together with Corollary \ref{corollary:truncated} shows that if we assume that $[\mathfrak{C}']$ carries the module $[\mathbb{T}']$, then $[\mathfrak{C}]\cong J_{W_I}^W( [\mathfrak{C}'] \otimes  \textup{sgn}_l )$ carries the module $[\mathbb{T}].$
\end{proof}

\begin{corollary} Consider a Weyl group of type $B_n$
with a weight function $L$ and parameter $s$ defined as above.  If statements {\bf P1-P15} hold, then
    \begin{enumerate}
        \item When $s \not\in \mathbb{N}$, the Kazhdan-Lusztig left cells of weight $s$ coincide
                with the irreducible combinatorial left cells of rank $\lfloor s \rfloor$.
        \item  When $s \in \mathbb{N}$, the Kazhdan-Lusztig left cells of weight $s$ coincide
                with the reducible combinatorial left cells of rank $s-1$.
    \end{enumerate}
If the set $\mathbb{T}$ of standard domino tableaux represents the left cell $\mathfrak{C}$ in $W_n$, then $[\mathfrak{C}] \cong [\mathbb{T}]$ as $W_n$-modules.  Furthermore, if $T \in \mathbb{T}$, then the number of elements of $\mathfrak{C}$ with right tableau $T$ is the dimension of the irreducible constituent $[p_T]$ of $[\mathfrak{C}]$.
\end{corollary}

\begin{proof}  The first part in the case $s \not\in \mathbb{N}$ is a result of C.~Bonnaf\'e \cite{bonnafe:knuth}.  To verify it in the case $s \in \mathbb{N}$, write a Kazhdan-Lusztig left cell $\mathfrak{C}$ in terms of combinatorial left cells as $\mathfrak{C}=\coprod_{i\in I} \mathfrak{D}_i$.  Since $[\mathfrak{C}]$ is constructible, the main result of \cite{pietraho:constructible} shows that
$[\mathfrak{C}] \cong [\widetilde{\mathbb{T}}]$ as $W_n$-modules where $\widetilde{\mathbb{T}} =\{MT(T, C) \; | \; C \subset OC^*(T)\}$ for some standard domino tableau $T$ of rank $r$.  Let each $\mathfrak{D}_i$ be represented by $\mathbb{T}_i = \{MT(T_i, C) \; | \; C \subset OC^*(T_i)\}$ and write $\mathbb{T} = \coprod_{i\in I} \mathbb{T}_i$.  By Theorem \ref{theorem:main}, $[\mathbb{T}] = [\widetilde{\mathbb{T}}].$  This implies that for every $i$, the set of beginning and ending squares of non-core open cycles in $T_i$ is contained in the corresponding set in $T$.  However, the size of this set is constant for every partition in the set of possible shapes of tableaux in $[\mathbb{T}]$.  By Lemma \ref{lemma:disjointshapes}, the only way this can occur is if $|I|=1$, that is, $\mathfrak{C}$ consists of a single combinatorial cell.

Finally, we verify the last statement of the corollary.  If $s \not \in \mathbb{N}$, consider a left cell $\mathfrak{C}$ represented by the tableau $T$.  Then $\dim [\mathfrak{C}] = \sum |\mathfrak{C} \cap \mathfrak{C'}^{-1}|$, the sum taken over all left cells $\mathfrak{C}'$ in $W_n$.  But $|\mathfrak{C} \cap \mathfrak{C'}^{-1}|=1$ iff the shape of the tableaux representing $\mathfrak{C}$ and $\mathfrak{C}'$ are the same; otherwise it is zero.  Since each left cell is represented by a unique tableau, the above sum equals the number of tableaux of the same shape as $T$.  This is the same as the number of elements of $\mathfrak{C}$ with right tableau $T$.  If $s \in \mathbb{N}$, consider left cells $\mathfrak{C}$ and $\mathfrak{C}'.$  For $w \in \mathfrak{C} \cap \mathfrak{C'}^{-1},$
$[shape \, T_r(w)]$ must be a component of both $[\mathfrak{C}]$ and $[\mathfrak{C'}]$.  Furthermore, each $w \in \mathfrak{C} \cap \mathfrak{C'}^{-1}$ must have the right tableau of a unique shape, establishing a bijection between $\mathfrak{C} \cap \mathfrak{C'}^{-1}$ and the set of irreducible modules common to $[\mathfrak{C}]$ and $[\mathfrak{C'}].$  If we let $\mathfrak{C}'$ vary over all left cells of $W_n$, the statement follows by Lemma \ref{lemma:12.15}.

\end{proof}

It should be remarked that the above statement classifying the module structure of left cells is not the strongest one could hope for.  In the so-called ``asymptotic" case when $s$ is sufficiently large,  M.~Geck has shown that whenever the tableaux representing $[\mathfrak{C}]$ and $[\mathfrak{C}']$ equal, then not only are the underlying $\mathcal{H}$-modules isomorphic, but the underlying structure constants are the same.  More precisely, there is a bijection $\mathfrak{C} \rightarrow \mathfrak{C}'$ sending $x \mapsto x'$ such that
$$h_{w,x,y}=h_{w,x',y'} \text{ for all $w \in W_n$ and $x,y \in \mathfrak{C}$}.$$
It would be interesting to know under what circumstances this stronger statement holds for other values of $s$.

\end{document}